\documentclass[11pt,letterpaper,reqno]{amsart}

\usepackage[T1]{fontenc}
\usepackage{amsmath,amssymb,amsthm,amsfonts,mathtools}
\usepackage{enumitem}
\usepackage{booktabs}
\usepackage{microtype}
\usepackage{doi}
 \usepackage{bookmark}
 \usepackage{hyperref}
 \hypersetup{pdfstartview={FitH},colorlinks=true,linkcolor=blue,citecolor=blue,urlcolor=blue}
\providecommand{\doi}[1]{doi:\,\texttt{#1}}

\usepackage{array}
\usepackage{tabularx}

\newtheorem{thm}{Theorem}[section]
\newtheorem{lem}[thm]{Lemma}
\newtheorem{prop}[thm]{Proposition}
\newtheorem{cor}[thm]{Corollary}

\newtheorem{conj}[thm]{Conjecture}
\theoremstyle{definition}

\newtheorem{rem}[thm]{Remark}

\numberwithin{equation}{section}

\DeclareMathOperator{\diag}{diag}
\DeclareMathOperator{\Ree}{Re}
\DeclareMathOperator{\Imm}{Im}

\newcommand{\HH}{\mathcal H}
\newcommand{\AD}{\mathcal A}

\begin{document}

\title[Sharp condition-number bounds for Higham matrices]
{Sharp condition-number bounds for growth factors of Higham matrices in Gaussian elimination}

\author[T.~Zhang]{Teng Zhang}

\address{School of Mathematics and Statistics, Xi'an Jiaotong University, Xi'an 710049, P. R. China}
\email{teng.zhang@stu.xjtu.edu.cn}



\dedicatory{Dedicated to the memory of Professor Nicholas J.~Higham.}

\begin{abstract}
	Higham's conjecture on the growth factor of complex symmetric positive
	definite matrices is a longstanding problem in the stability theory of
	Gaussian elimination without pivoting.  It asserts that  every complex matrix \(A=B+iC\) with \(B\) and \(C\) real symmetric positive
	definite, is  called Higham matrix and has growth factor $\rho_n(A)<2$. In 2013, Drury [Linear Algebra Appl. \textbf{439} (2013), no.~10, 3129--3133] proved that \(\rho_n(A)\le 2\). In fact, we will see his sectorial determinant method can be refined to
	give the strict bound \(\rho_n(A)<2\) for each fixed Higham matrix; however,
	the resulting constant \(1+\delta_A^2\) depends on the matrix \(A\).  In this
	paper, we establish sharp condition-number-dependent lower and upper bounds
	for the growth factors of Higham matrices, thereby providing a quantitative
	refinement of Drury's result.  The main ingredient is a sharp scalar
	Schur-complement inequality, proved via a two-dimensional domination.
	We also obtain corresponding sharp scalar and diagonal estimates for
	 accretive-dissipative matrices, and an improved
	entrywise growth bound for that broader class.
\end{abstract}

\maketitle
\tableofcontents
\section{Introduction}\label{sec:intro} Consider the linear system
\[
Ax=b,\qquad A\in\mathbb C^{n\times n}\text{ with } n\ge 2.
\]
After \(k\) steps of Gaussian elimination without pivoting, where
\(1\le k\le n-1\), the transformed system has the block form
\[
\begin{bmatrix}
	* & *\\
	0 & A^{(k)}
\end{bmatrix}x=b^{(k)}.
\]
Here
$
A^{(k)}=(a_{ij}^{(k)})_{i,j=1}^{n-k}
\in\mathbb C^{(n-k)\times(n-k)}
$
is the active matrix remaining after the first \(k\) elimination steps.
  The  growth factor of $A$ in Gaussian elimination is defined by
\begin{equation}\label{eq:post-growth}
       \rho_n(A)=
       \frac{\max_{1\le k\le n-1}\max_{i,j}|a_{ij}^{(k)}|}
            {\max\limits_{i,j}|a_{ij}|}.
\end{equation}
Wilkinson's backward error analysis makes such growth quantities central to the stability of Gaussian elimination: when no pivoting is used, the computed factors are reliable only if the elements appearing during elimination remain controlled relative to the original entries; see, for instance, the discussion in \cite[Chapter~9]{HighamBook}.  For several classical classes this control is automatic: Hermitian positive definite matrices and totally nonnegative matrices have standard growth factor $1$, while row or column diagonally dominant matrices satisfy a small uniform bound.  Higham's class was introduced as a non-Hermitian, complex symmetric class in which comparably strong stability properties might still persist; see \cite{Higham1998}.  This is precisely the class studied in this paper.

\subsection{Higham's conjecture for the growth factors of Higham matrices}
Higham \cite[p.~1592]{Higham1998} introduced the notion of a complex symmetric positive definite (CSPD) matrix. Specifically, a complex symmetric matrix
\[
A = B + iC,\qquad B,C\in\mathbb R^{n\times n},
\]
is said to be CSPD if both \(B\) and \(C\) are real symmetric positive definite. Such matrices are also referred to as \emph{Higham matrices}. We denote the set of all Higham matrices of order \(n\) by \(\HH_n\), that is,
\[
\HH_n := \{P+iQ : P,Q\in\mathbb R^{n\times n}\text{ are symmetric positive definite}\}.
\]

  Higham \cite[Lemmas~2.1--2.3]{Higham1998} proved the basic structural facts needed for Gaussian elimination: Higham matrices are nonsingular, their leading principal submatrices are nonsingular, and Schur complements of principal blocks are again Higham.  Higham  \cite[Lemma~2.4]{Higham1998} also recorded the crucial diagonal dominance in modulus: for $A=(a_{ij})\in \HH_n$,
\[
       |a_{ij}|<\sqrt{|a_{ii}|\,|a_{jj}|}\qquad (i\ne j),
\]
so the largest entry in modulus lies on the diagonal.  These facts reduce the analysis of element growth to the behavior of scalar diagonal Schur complements. Later in \cite[Theorem~2.5]{Higham1998},
Higham claimed to have proved the following result, now known as ``Higham's conjecture''.  
\begin{conj}[Higham]\label{conj:Higham}
	Let $A\in \HH_n$. Then 
	$$\rho_n(A)<2,$$
	where the constant $2$ is sharp.
\end{conj} We note that his sharpness example in \cite[Theorem~2.5]{Higham1998}  is valid. It involves a singular limiting \(2\times2\) matrix, regularized by adding \(\varepsilon(1+i)I\), and shows that the class supremum is \(2\), although no fixed Higham matrix on the boundary is admissible.  The proof of \cite[Theorem~2.5]{Higham1998}, however, contains an invalid step: it squares one-sided real and imaginary inequalities without first having the corresponding lower bounds.  Higham later noted in his book that the proof of the growth-factor bound in  \cite[Theorem~2.5]{Higham1998} is incorrect and posed the sharp bound as an open research problem; see \cite[p.~210]{HighamBook} and \cite[p.~212, Problem~10.12]{HighamBook}.

 George, Ikramov and Kucherov \cite{GeorgeIkramovKucherov2002} introduced and studied the generalized Higham class $A=B+i C$ with $B$ and $C$ Hermitian positive definite. 
 Such matrices are also called  \emph{generalized Higham matrices} or \emph{accretive-dissipative matrices}. We denote the set of all generalized Higham matrices of order \(n\) by \( \AD_n\), that is,
 \[
 \AD_n:=\{B+i C: B,C\in\mathbb C^{n\times n}\text{ are Hermitian positive definite}\}.
 \]  Clearly, $\HH_n\subset\AD_n$. For an invertible matrix \(X\), its spectral condition number is denoted by
$
 \kappa(X)=\|X\|_\infty\,\|X^{-1}\|_\infty,
$
 where \(\|\cdot\|_\infty\) denotes the spectral norm.
The subsequent development of Conjecture~\ref{conj:Higham} can be summarized as follows; a good survey see \cite{Zhang2014MatrixDecomposition}.
Table~\ref{tab:higham-history} concerns the real-symmetric Higham class
$\HH_n$
whereas Table~\ref{tab:ad-history} concerns the larger accretive-dissipative class $\AD_n$.

\begin{table}[htbp]
	\centering
	\renewcommand{\arraystretch}{1.18}
	\begin{tabularx}{\textwidth}{@{}
			>{\centering\arraybackslash}p{0.25\textwidth}
			>{\raggedright\arraybackslash}X
			>{\raggedright\arraybackslash}p{0.31\textwidth}
			@{}}
		\toprule
		Upper bound for \(\rho_n(A)\) & Hypotheses & Reference \\
		\midrule
		\(\displaystyle \frac{1+\sqrt{17}}{4}\)
		&
		Buckley's normalized subclass \(A=I+iC\), \(C=C^T\) positive definite
		&
		Ikramov--Kucherov \cite{IkramovKucherov2000}
		\\[1mm]
		
		\(<3\)
		&
		\(A\in\HH_n\)
		&
		George--Ikramov--Kucherov \cite{GeorgeIkramovKucherov2002}
		\\[1mm]
		
		bound tending to \(1\) as \(\alpha\downarrow0\)
		&
		\(A=B+iC\in\HH_n\) with \(C\le \alpha B\)
		&
		George--Ikramov \cite{GeorgeIkramov2004}
		\\[1mm]
		
		\(<2\sqrt2\)
		&
		\(A\in\HH_n\)
		&
		Lin \cite{Lin2014}
		\\[1mm]
		
		\(\displaystyle
		\sqrt2\left[
		1+\left(\frac{\omega-1}{\omega+1}\right)^2
		\right]\)
		&
		\(A\in\HH_n\), \(\omega=\max\{\kappa(B),\kappa(C)\}\)
		&
		Yang \cite{Yang2014}
		\\[1mm]
		
		\(< 2\)
		&
		\(A\in\HH_n\)
		&
	Drury \cite{Drury2013}; see Remark~\ref{rem:drury-determinant}
		\\[1mm]
		
		\(\displaystyle
		\frac{2(1+\omega^2)}{(1+\omega)^2}\)
		&
		\(A\in\HH_n\), \(\omega=\max\{\kappa(B),\kappa(C)\}\)
		&
		this paper, Theorem~\ref{thm:main}
		\\
		\bottomrule
	\end{tabularx}
	\caption{Representative upper bounds for the growth factor
		\(\rho_n(A)\) in the Higham class \(\HH_n\).}
	\label{tab:higham-history}
\end{table}

\begin{table}[htbp]
	\centering
	\renewcommand{\arraystretch}{1.18}
	\begin{tabularx}{\textwidth}{@{}
			>{\centering\arraybackslash}p{0.25\textwidth}
			>{\raggedright\arraybackslash}X
			>{\raggedright\arraybackslash}p{0.31\textwidth}
			@{}}
		\toprule
		Upper bound for \(\rho_n(A)\) & Hypotheses & Reference \\
		\midrule
		\(<3\sqrt2\)
		&
		\(A\in\AD_n\)
		&
		George--Ikramov--Kucherov \cite{GeorgeIkramovKucherov2002}
		\\[1mm]
		
		\(<4\)
		&
		\(A\in\AD_n\)
		&
		Lin \cite{Lin2014}
		\\[1mm]
		
		\(\displaystyle
		2\left[
		1+\left(\frac{\omega-1}{\omega+1}\right)^2
		\right]\)
		&
		\(A\in\AD_n\), \(\omega=\max\{\kappa(B),\kappa(C)\}\)
		&
		Yang \cite{Yang2014}
		\\[1mm]
		
		\(<2\sqrt2\)
		&
		\(A\in\AD_n\)
		&
		Drury
		\cite{Drury2013}; see Remark~\ref{rem:drury-determinant}
		\\[1mm]
		
		\(\displaystyle
		\frac{2\sqrt{2}(1+\omega^2)}{(1+\omega)^2}\)
		&
		\(A\in\AD_n\), \(\omega=\max\{\kappa(B),\kappa(C)\}\)
		&
		this paper, Theorem~\ref{thm:generalized}
		\\
		\bottomrule
	\end{tabularx}
	\caption{Representative upper bounds for the  growth factor
		\(\rho_n(A)\) in the accretive-dissipative class \(\AD_n\).}
	\label{tab:ad-history}
\end{table}
\begin{rem}\label{rem:drury-determinant}
	Drury's result in \cite[Immediately after Theorem~1.3]{Drury2013} is formulated primarily as a sectorial
	Fischer determinantal inequality.  Although Drury implicitly stated only the consequence
	\(\rho_n(A)\le 2\), his method can in fact be refined to yield the strict
	bound \(\rho_n(A)<2\) for every fixed Higham matrix $A$. To avoid any possible ambiguity about how
	this determinant inequality implies the quoted growth-factor bound, we put the argument in Appendix~\ref{app:drury}.
\end{rem}

Tables~\ref{tab:higham-history} and~\ref{tab:ad-history} show the two main
features of the problem.  First, the Higham class enjoys a special
diagonal-maximality property, which allows scalar diagonal Schur-complement
bounds to control the full entrywise growth.  Second, this property is lost in
the accretive-dissipative class, where an additional diagonal-to-entry factor is
unavoidable in the standard reduction. This paper gives a direct Schur-complement proof of sharp
condition-number-dependent bounds and, as a consequence, recovers the sharp
uniform Higham bound. We also obtain the corresponding
sharp scalar and diagonal estimates in the larger accretive-dissipative class.

\subsection{Gaps in Guo--Gu--Li's paper}

A recent paper of Guo, Gu, and Li \cite{GuoGuLi} states the
condition-number estimate
\[
\frac{4\omega}{(1+\omega)^2}
\le \rho_n(A)
\le \frac{2(1+\omega^2)}{(1+\omega)^2},
\qquad
A=B+iC\in\AD_n,\qquad
\omega=\max\{\kappa(B),\kappa(C)\}.
\]
Guo, Gu, and Li~\cite{GuoGuLi} then claimed to have resolved
Conjecture~\ref{conj:Higham}.  However, the first proof of Conjecture~\ref{conj:Higham}
was already given by Drury~\cite{Drury2013}, and this priority issue is not
reflected in their paper.
Their first arXiv version of that work appeared in 2013 \cite{GuoGuLi}.  After a gap in that
version was pointed out by Drury \cite[p. 3131, line 5--line 7]{Drury2013}, the authors posted a revised version in 2025 \cite{GuoGuLi}
and asserted \cite[Acknowledgment]{GuoGuLi} that the gap had been fixed.  We show below that the
same gap remains.    Although the paper has recently appeared online in \emph{Mathematics of
	Computation}, the following three points show why its condition-number proof
cannot be used as stated. That is, their proof is invalid. We refer to the arXiv v2 version when identifying
the specific gaps in the argument.

\subsubsection*{Gap 1: the componentwise-to-modulus step is invalid}
In the notation of a scalar Schur complement, write
\[
q=\beta^T A_k^{-1}\alpha.
\]
In \cite[p. 7, begin with ``With the help of ...'']{GuoGuLi}, they obtained inequalities of the type
\[
\Ree q\le M_1,
\qquad
\Imm q\le M_2,
\]
with $M_1,M_2\ge0$, and then uses them as if they implied
\[
|q|\le \sqrt{M_1^2+M_2^2}.
\]
This implication is false unless one also knows lower bounds $\Ree q\ge -M_1$ and $\Imm q\ge -M_2$.

The failure occurs inside the accretive-dissipative class.  Let $0<r<1$ and consider the $2\times2$ block with leading scalar $1+i$,
\[
B=\begin{bmatrix}1&r\\ r&1\end{bmatrix},
\qquad
C=\begin{bmatrix}1&i r/2\\ -i r/2&1\end{bmatrix}.
\]
Then $B$ and $C$ are Hermitian positive definite.  For the scalar update one has
\[
\begin{aligned}
	q&=(r+i(-i r/2))(1+i)^{-1}(r+i(i r/2))  
	=\frac{3r^2}{8}(1-i).
\end{aligned}
\]
Here
\[
c^*C_{11}^{-1}c=\frac{r^2}{4},
\qquad
\Imm q=-\frac{3r^2}{8}.
\]
Thus $\Imm q\le r^2/4$ is true but
\[
(\Imm q)^2=\frac{9r^4}{64}>\frac{r^4}{16}
=\left(c^*C_{11}^{-1}c\right)^2.
\]
So the subsequent squaring step has no logical basis. This invalidates all the subsequent arguments.

\subsubsection*{Gap 2: diagonal maximality is not available for generalized Higham matrices}
For Higham matrices, \cite[Lemma~2.2]{GuoGuLi} says that every active matrix has its largest entry on the diagonal.  This is a special real-symmetric feature.  It is false for accretive-dissipative matrices.

Indeed, for $0<r<1$ let
\[
B_r=\begin{bmatrix}1&r\\ r&1\end{bmatrix},
\qquad
C_r=\begin{bmatrix}1&-i r\\ i r&1\end{bmatrix}.
\]
Then
\[
A_r=B_r+i C_r
=\begin{bmatrix}1+i&2r\\0&1+i\end{bmatrix}
\]
is accretive-dissipative.  If $r>1/\sqrt2$, then
\[
|(A_r)_{12}|=2r>\sqrt2=|(A_r)_{11}|=|(A_r)_{22}|.
\]
Thus any proof for generalized Higham matrices must control off-diagonal active entries directly or pay an additional entry-versus-diagonal factor.  A diagonal-only argument is sufficient for Higham matrices but not for the whole accretive-dissipative class.

\subsubsection*{Gap 3: the lower growth-factor bound depends on the convention}
The displayed lower bound
\[
\frac{4\omega}{(1+\omega)^2}\le\rho_n(A)
\]
is harmless but non-sharp if $\rho_n(A)$ denotes the standard growth factor including $k=0$: since $\omega\ge1$,
\[
0<\frac{4\omega}{(1+\omega)^2}\le1,
\]
and the inclusion of $A^{(0)}=A$ gives the trivial sharp lower bound $1$.

Under the post-elimination convention \eqref{eq:post-growth}, however, the same displayed lower bound is generally false for $\omega>1$.  The sharp universal lower bound is $1/\omega$, as proved in Theorem~\ref{thm:main}.  The diagonal example
\[
A=(1+i)\diag(\omega,1,\ldots,1)
\]
has $\kappa(\Ree A)=\kappa(\Imm A)=\omega$ and satisfies
\[
\rho_n(A)=\frac1\omega
<\frac{4\omega}{(1+\omega)^2}
\qquad (\omega>1).
\]

\subsection{Sharp condition-number bounds for $\rho_n$} 
We now state the main result of the paper.   The first theorem gives sharp two-sided
condition-number bounds for Higham matrices.
\begin{thm}\label{thm:main}
	Let $A=B+i C\in\HH_n$ and suppose that
$
	\max\{\kappa(B),\kappa(C)\}\le \omega .
$
	For each elimination stage $1\le k\le n-1$ set
	\[
	\rho_{n,k}(A):=
	\frac{\max_{i,j}|a_{ij}^{(k)}|}{\max_{i,j}|a_{ij}|}.
	\]
	Then Gaussian elimination without pivoting is well defined and
	\begin{equation}\label{eq:stage-bound}
		\frac1\omega\le \rho_{n,k}(A)
		\le \frac{2(1+\omega^2)}{(1+\omega)^2},
		\qquad 1\le k\le n-1 .
	\end{equation}
	Consequently the  growth factor \eqref{eq:post-growth} satisfies
	\begin{equation}\label{eq:main-bound}
		\frac1\omega\le \rho_n(A)
		\le \frac{2(1+\omega^2)}{(1+\omega)^2}.
	\end{equation}
\end{thm}

Both constants in Theorem~\ref{thm:main} are sharp in the following sense.
 The lower
constant is already forced by diagonal matrices, while the upper constant is
attained by a two-dimensional extremal family.
\begin{thm}\label{thm:sharp}
	Let $1\le\omega<\infty$.
	\begin{enumerate}[label=\textup{(\roman*)}]
		\item For every $n\ge2$, the lower constant $1/\omega$ in \eqref{eq:stage-bound} is attained at every stage $1\le k\le n-1$ by a diagonal Higham matrix satisfying
		\[
		\kappa(B)=\kappa(C)=\omega .
		\]
		Thus no larger universal lower bound is possible, even under the exact condition-number requirement.
		\item There is a $2\times2$ Higham matrix $A_\omega=B_\omega+i C_\omega$ such that
		\[
		\kappa(B_\omega)=\kappa(C_\omega)=\omega
		\quad\text{and}\quad
		\rho_2(A_\omega)=\frac{2(1+\omega^2)}{(1+\omega)^2} .
		\]
		Consequently the upper constant in \eqref{eq:main-bound} cannot be decreased.
	\end{enumerate}
\end{thm}
We next record the corresponding result for the larger accretive-dissipative
class.  Since the diagonal maximality property is no longer available in
\(\AD_n\), the passage from diagonal Schur-complement bounds to full entrywise
growth costs the standard factor \(\sqrt2\).  The lower bound, however, remains
sharp.
\begin{thm}\label{thm:generalized}
	Let $A=B+i C\in\AD_n$ and suppose that
	$
	\max\{\kappa(B),\kappa(C)\}\le\omega .
	$
	Then 
	\begin{equation*}
		\frac1\omega
		\le \rho_n(A)
		\le \frac{2\sqrt2(1+\omega^2)}{(1+\omega)^2}.
	\end{equation*}
	The lower constant $1/\omega$ is sharp, even for diagonal Higham matrices.  In particular,  $\rho_n(A)< 2\sqrt2$.
\end{thm}
The factor \(\sqrt2\) in the upper bound of Theorem~\ref{thm:generalized}
comes from a general diagonal-to-entry estimate for accretive-dissipative
matrices.  Numerical experiments suggest that this factor may be an artifact
of the reduction rather than a genuine obstruction.  This leads to the following
conjecture.
\begin{conj}\label{conj:me}
		Let $A=B+i C\in\AD_n$ and suppose that
	$
	\max\{\kappa(B),\kappa(C)\}\le\omega .
	$
	Then $$\rho_n(A)
	\le \frac{2(1+\omega^2)}{(1+\omega)^2},$$ where the constant $\frac{2(1+\omega^2)}{(1+\omega)^2}$  is sharp.
\end{conj} 
The sharpness of the constant in Conjecture~\ref{conj:me} follows from Theorem~\ref{thm:sharp}(ii), since the
extremal Higham matrices constructed there belong to the subclass
\(\HH_n\subset\AD_n\).
\subsection*{Sketch of the proofs.}

The key estimate is local: for a principal block
\[
        \begin{bmatrix}G&z\\ z^T&\alpha\end{bmatrix}
\]
of a Higham matrix whose real and imaginary parts both have condition number at most \(\omega\), we prove
\[
       |z^TG^{-1}z|\le \left(\frac{\omega-1}{\omega+1}\right)^2|\alpha|.
\]
Consequently the scalar Schur complement $\sigma=\alpha-z^TG^{-1}z$ satisfies the sharp two-sided estimate
\[
       \frac{4\omega}{(1+\omega)^2}|\alpha|
       \le |\sigma|
       \le \frac{2(1+\omega^2)}{(1+\omega)^2}|\alpha|.
\]
The proof of this scalar inequality is the key part of the paper.  It rests on a two-dimensional domination lemma that converts the phase of the scalar update into a positive semidefinite $2\times2$ majorization.  Once this local estimate is available, the global growth-factor bound follows from Higham's hereditary and diagonal-maximality lemmas.

\subsection*{Organization of this paper.}
The paper is organized as follows.  Section~\ref{sec:prelim} recalls the standard structural and Kantorovich facts.  Section~\ref{sec:twodim} proves the key two-dimensional domination lemma and derives the sharp scalar Schur-complement bounds.  Section~\ref{sec:main-proof} proves Theorems~\ref{thm:main} and \ref{thm:sharp}.  Section~\ref{sec:ad} records the corresponding scalar result for accretive-dissipative matrices, proves Theorem~\ref{thm:generalized} and explains why full entrywise growth there loses a factor $\sqrt2$.

\section{Preliminaries}\label{sec:prelim}

\subsection{The Wielandt--Kantorovich block estimate}
The following Wielandt--Kantorovich block estimate is the only place where the condition number enters.

\begin{lem}\label{lem:kantorovich-block}
Let
$
      H=\begin{bmatrix}H_{11}&h\\ h^*&\eta\end{bmatrix}\in \mathbb{C}^{n\times n}
$ be Hermitian positive definite
and  $\kappa(H)\le\omega$.  Then
\begin{equation*}
      h^*H_{11}^{-1}h\le \left(\frac{\omega-1}{\omega+1}\right)^2\eta,
\end{equation*}
where the constant $\left(\frac{\omega-1}{\omega+1}\right)^2$ is sharp.
\end{lem}

\begin{proof}
The block form is the scalar trailing-block case of the Wielandt--Kantorovich inequality in Zhang \cite[p.~277, Eq.~(6)]{Zhang2001}; the same inequality is also recorded explicitly in \cite[Section~2, Lemma~4]{XueHu2015}.  Sharpness is already attained by $\begin{bmatrix}1&\frac{\omega-1}{\omega+1}\\ \frac{\omega-1}{\omega+1}&1\end{bmatrix}$.
\end{proof}

\subsection{Structural facts for Higham matrices}
We need two hereditary properties and one diagonal maximality property. 
\begin{lem}[{\cite[Lemmas~2.1--2.3]{Higham1998}}]\label{lem:inverse-schur}
Let $G\in\HH_n$. Then $G$ is nonsingular, every principal submatrix of $G$ and every Schur complement of a nonsingular principal block of $G$ again belongs to the corresponding Higham class.
\end{lem}

\begin{lem}[{\cite[Lemma~2.4]{Higham1998}}]\label{lem:diagonal-max}
If $G=(g_{ij})=P+i Q\in\HH_n$, then for every $i\ne j$,
\begin{equation*}
        |g_{ij}|^2<|g_{ii}|\,|g_{jj}|.
\end{equation*}
In particular,
\[
        \max_{i,j}|g_{ij}|=\max_i |g_{ii}|.
\]
\end{lem}

We denote by \(\le\) ($<$) the usual Loewner order on Hermitian matrices. For a Hermitian matrix \(H\), we write \(\lambda_{\min}(H)\) and \(\lambda_{\max}(H)\) for its smallest and largest eigenvalues, respectively. We next establish a Loewner lower bound for the Schur complements of matrices in the generalized Higham class \(\AD_n\).
\begin{lem}\label{lem:schur-loewner-lower}
Let $A=B+i C\in\AD_n$ be partitioned as
\[
       A=\begin{bmatrix}A_{11}&A_{12}\\ A_{21}&A_{22}\end{bmatrix},
       \qquad
       B=\begin{bmatrix}B_{11}&B_{12}\\ B_{21}&B_{22}\end{bmatrix},
       \qquad
       C=\begin{bmatrix}C_{11}&C_{12}\\ C_{21}&C_{22}\end{bmatrix},
\]
where $A_{11}$ is a nonsingular principal block.  If
\[
       S=A/A_{11}=A_{22}-A_{21}A_{11}^{-1}A_{12}=R+i T
\]
is the corresponding Schur complement, with $R$ and $T$ Hermitian, then
\begin{equation*}
       R\ge B/B_{11},
       \qquad
       T\ge C/C_{11}.
\end{equation*}
Consequently, if $m_B=\lambda_{\min}(B)$ and $m_C=\lambda_{\min}(C)$, then
\[
       R\ge m_B I,
       \qquad
       T\ge m_C I .
\]
\end{lem}

\begin{proof}
Fix a vector $x$ of compatible size and put $y=-A_{11}^{-1}A_{12}x$.  Then
\[
       \begin{bmatrix}y\\x\end{bmatrix}^{\!*}
       A
       \begin{bmatrix}y\\x\end{bmatrix}
       =x^*Sx.
\]
Taking real parts gives
\[
       x^*Rx=\begin{bmatrix}y\\x\end{bmatrix}^{\!*}
       B
       \begin{bmatrix}y\\x\end{bmatrix}
       \ge \min_w
       \begin{bmatrix}w\\x\end{bmatrix}^{\!*}
       B
       \begin{bmatrix}w\\x\end{bmatrix}
       =x^*(B/B_{11})x.
\]The last equality is the standard variational characterization of the Schur
complement\footnote{Indeed, completing the square gives
	\[
	\begin{aligned}
		\begin{bmatrix}w\\x\end{bmatrix}^{\!*}
		B
		\begin{bmatrix}w\\x\end{bmatrix}
		&=
		\bigl(w+B_{11}^{-1}B_{12}x\bigr)^*
		B_{11}
		\bigl(w+B_{11}^{-1}B_{12}x\bigr) +
		x^*\bigl(B_{22}-B_{21}B_{11}^{-1}B_{12}\bigr)x .
	\end{aligned}
	\]
	Since \(B_{11}>0\), the first term is nonnegative and vanishes at
	\(w=-B_{11}^{-1}B_{12}x\).  Hence the minimum equals
	\(x^*(B/B_{11})x\).
	This proves $R\ge B/B_{11}$.}.    Taking imaginary parts and minimizing the corresponding quadratic form for $C$ gives $T\ge C/C_{11}$.

Finally, if $H>0$ and $m=\lambda_{\min}(H)$, then every principal Schur complement $H/H_{11}$ satisfies $H/H_{11}\ge mI$.  Indeed, $(H/H_{11})^{-1}$ is a principal submatrix of $H^{-1}$, and $H^{-1}\le m^{-1}I$.  Applying this to $B$ and $C$ gives the last assertion.
\end{proof}

\section{Sharp scalar estimate for Higham blocks}\label{sec:twodim}
The scalar Schur complement estimate is reduced to the following elementary two-dimensional inequality.  This is the key technical lemma of the paper.  The formulation is slightly more general than is needed for Higham matrices $\HH_n$, because it also covers the accretive-dissipative case $\AD_n$ in Section~\ref{sec:ad}.

For $\phi\in[-\pi,\pi]$ define
\begin{equation*}
\vartheta(\phi)=
\begin{cases}
-\phi/3, & -\pi\le \phi<0,\\[1mm]
\phi, & 0\le \phi\le \pi/2,\\[1mm]
(2\pi-\phi)/3, & \pi/2<\phi\le \pi .
\end{cases}
\end{equation*}
Then $\vartheta(\phi)\in[0,\pi/2]$.
\begin{lem}\label{lem:twodim}
For every $d>0$, every $\phi\in[-\pi,\pi]$, and all $s,t\in\mathbb C$,
\begin{equation}\label{eq:twodim}
\Ree\left(e^{-i\phi}\frac{(\overline{s}+i\overline{t})(s+i t)}{1+i d}\right)
\le
\cos\vartheta(\phi)\,|s|^2+\sin\vartheta(\phi)\,\frac{|t|^2}{d} .
\end{equation}
In particular, when $s,t\in \mathbb{R}$,
\begin{equation*}
\Ree\left(e^{-i\phi}\frac{(s+i t)^2}{1+i d}\right)
\le
\cos\vartheta(\phi)\,s^2+\sin\vartheta(\phi)\,\frac{t^2}{d} .
\end{equation*}
\end{lem}

\begin{proof}
Put $w=(s,t/\sqrt d)^T$.  A direct expansion shows that the left-hand side of \eqref{eq:twodim} is $w^*H_{\phi,d}w$, where
\[
H_{\phi,d}=\frac{1}{1+d^2}
\begin{bmatrix}
\cos\phi-d\sin\phi & \sqrt d\,(d\cos\phi+\sin\phi)\\[1mm]
\sqrt d\,(d\cos\phi+\sin\phi) & d(d\sin\phi-\cos\phi)
\end{bmatrix}.
\]
It is enough to prove
\[
K_{\phi,d}:=
\begin{bmatrix}
\cos\vartheta(\phi)&0\\
0&\sin\vartheta(\phi)
\end{bmatrix}
-H_{\phi,d}\ge0.
\]

If $0\le\phi\le\pi/2$, then $\vartheta(\phi)=\phi$ and
\[
K_{\phi,d}=\frac{d\cos\phi+\sin\phi}{1+d^2}
\begin{bmatrix}
 d&-\sqrt d\\
 -\sqrt d&1
\end{bmatrix}\ge0.
\]

It remains to consider $\phi\in[-\pi,0)\cup(\pi/2,\pi]$.  Then, with $\vartheta=\vartheta(\phi)$,
\[
        \cos\phi=\cos3\vartheta,
        \qquad
        \sin\phi=-\sin3\vartheta,
        \qquad
        0<\vartheta\le\pi/2 .
\]
A calculation gives
\begin{equation}\label{eq:k11}
(1+d^2)(K_{\phi,d})_{11}
 =\cos\vartheta\,d^2-\sin3\vartheta\,d+\cos\vartheta-\cos3\vartheta
\end{equation}
and
\begin{equation}\label{eq:kdet}
(1+d^2)\det K_{\phi,d}
 =4\sin\vartheta\cos\vartheta\,(
   \sin\vartheta-d\cos\vartheta)^2\ge0 .
\end{equation}
We prove $(K_{\phi,d})_{11}>0$.  If $0<\vartheta\le\pi/3$, then the right-hand side of \eqref{eq:k11} is a quadratic polynomial in $d$ with leading coefficient $\cos\vartheta>0$.  Its discriminant $\Delta$ satisfies
\[
-\Delta
=4\cos\vartheta(\cos\vartheta-\cos3\vartheta)-\sin^2 3\vartheta
=\sin^2\vartheta\bigl(7+8\sin^2\vartheta-16\sin^4\vartheta\bigr)>0,
\]
because $0<\sin^2\vartheta\le3/4$.  Thus the quadratic is positive for every $d>0$.  If $\pi/3\le\vartheta\le\pi/2$, then $\sin3\vartheta\le0$ and the expression in \eqref{eq:k11} is positive for every $d>0$.

Therefore $(K_{\phi,d})_{11}>0$ and, by \eqref{eq:kdet}, $\det K_{\phi,d}\ge0$.  Hence $K_{\phi,d}\ge0$.
\end{proof}
Next, we prove the following sharp scalar estimate for Higham blocks.

\begin{prop}\label{prop:scalar-higham}
Let
$$
M=\begin{bmatrix}G&z\\ z^T&\alpha\end{bmatrix}\in\HH_{m+1}
=
\begin{bmatrix}P&u\\ u^T&\beta\end{bmatrix}
+i
\begin{bmatrix}Q&v\\ v^T&\gamma\end{bmatrix},
$$
where $G=P+i Q$, $z=u+i v$, and $\alpha=\beta+i\gamma$.  Suppose that the real and imaginary parts of $M$ have condition numbers at most $\omega$.  Then
\begin{equation}\label{eq:sharp-scalar-q}
       |z^T G^{-1}z|\le \left(\frac{\omega-1}{\omega+1}\right)^2 |\alpha|.
\end{equation}
Consequently, the scalar Schur complement
$
       \sigma=\alpha-z^T G^{-1}z
$
satisfies the two-sided bound
\begin{equation}\label{eq:sharp-scalar-sigma}
       \frac{4\omega}{(1+\omega)^2}|\alpha|
       \le |\sigma|
       \le \frac{2(1+\omega^2)}{(1+\omega)^2}|\alpha| .
\end{equation}
The constant $\left(\frac{\omega-1}{\omega+1}\right)^2$ in \eqref{eq:sharp-scalar-q} and both endpoints in \eqref{eq:sharp-scalar-sigma} are sharp.
\end{prop}

\begin{proof} By simultaneous congruence diagonalization of two real symmetric
	positive definite matrices, we can
choose a real nonsingular matrix $R$ such that
\[
       R^T P R=I,
       \qquad
       R^T Q R=D=\diag(d_1,\ldots,d_m),
       \qquad d_j>0.
\]
Write
\[
       R^Tz=x+i y,
\]
where $x=(x_1,\ldots,x_m)^T,y=(y_1,\ldots,y_m)^T\in\mathbb R^m.$
Then
\begin{equation}\label{eq:qsum-higham}
       q:=z^TG^{-1}z
       =\sum_{j=1}^m\frac{(x_j+i y_j)^2}{1+i d_j}.
\end{equation}
By Lemma~\ref{lem:kantorovich-block}, applied to the two real positive definite block matrices
\[
       \begin{bmatrix}P&u\\ u^T&\beta\end{bmatrix},
       \qquad
       \begin{bmatrix}Q&v\\ v^T&\gamma\end{bmatrix},
\]
we have
\begin{equation*}
       x^Tx=u^TP^{-1}u\le \left(\frac{\omega-1}{\omega+1}\right)^2\beta,
       \qquad
       y^TD^{-1}y=v^TQ^{-1}v\le \left(\frac{\omega-1}{\omega+1}\right)^2\gamma .
\end{equation*}
If $q=0$, \eqref{eq:sharp-scalar-q} is immediate.  Otherwise let $\phi=\arg q\in[-\pi,\pi]$.  Applying Lemma~\ref{lem:twodim} term by term to \eqref{eq:qsum-higham} gives
\[
\begin{aligned}
       |q|&=\Ree(e^{-i\phi}q)  \\
       &\le \cos\vartheta(\phi)\sum_{j=1}^m x_j^2
       +\sin\vartheta(\phi)\sum_{j=1}^m\frac{y_j^2}{d_j}  \\
       &\le \left(\frac{\omega-1}{\omega+1}\right)^2\bigl(\beta\cos\vartheta(\phi)+\gamma\sin\vartheta(\phi)\bigr)
       \le \left(\frac{\omega-1}{\omega+1}\right)^2\sqrt{\beta^2+\gamma^2}
       =\left(\frac{\omega-1}{\omega+1}\right)^2|\alpha|.
\end{aligned}
\]
This proves \eqref{eq:sharp-scalar-q}.  The two-sided estimate \eqref{eq:sharp-scalar-sigma} follows from the triangle and reverse triangle inequalities:
\[
       |\alpha|-|q|\le |\sigma|\le |\alpha|+|q|.
\]

It remains only to record sharpness.  Consider the two $2\times2$ Higham matrices
\begin{equation}\label{eq:scalar-extremal-families}
       A_\omega^-=
       \begin{bmatrix}
       1+i&\dfrac{\omega-1}{\omega+1}(1+i)\\[1mm]
       \dfrac{\omega-1}{\omega+1}(1+i)&1+i
       \end{bmatrix},
       \qquad
       A_\omega^+=
       \begin{bmatrix}
       1+i&\dfrac{\omega-1}{\omega+1}(1-i)\\[1mm]
       \dfrac{\omega-1}{\omega+1}(1-i)&1+i
       \end{bmatrix}.
\end{equation}
For $A_\omega^-$, the real and imaginary parts are both
\(
\bigl[\begin{smallmatrix}1&\frac{\omega-1}{\omega+1}\\[.5mm]
\frac{\omega-1}{\omega+1}&1\end{smallmatrix}\bigr]
\), so their condition numbers are $\omega$.  Moreover
\[
       \frac{\left(\frac{\omega-1}{\omega+1}\right)^2(1+i)^2}{1+i}
       =\left(\frac{\omega-1}{\omega+1}\right)^2(1+i),
\]
and hence
\[
       \sigma=\frac{4\omega}{(1+\omega)^2}(1+i).
\]
Thus equality holds in the lower endpoint of \eqref{eq:sharp-scalar-sigma}, and equality also holds in \eqref{eq:sharp-scalar-q}.  For $A_\omega^+$, the real and imaginary parts are
\(
\bigl[\begin{smallmatrix}1&\frac{\omega-1}{\omega+1}\\[.5mm]
\frac{\omega-1}{\omega+1}&1\end{smallmatrix}\bigr]
\)
and
\(
\bigl[\begin{smallmatrix}1&-\frac{\omega-1}{\omega+1}\\[.5mm]
-\frac{\omega-1}{\omega+1}&1\end{smallmatrix}\bigr]
\), again with condition numbers $\omega$, and
\[
       \frac{\left(\frac{\omega-1}{\omega+1}\right)^2(1-i)^2}{1+i}
       =-\left(\frac{\omega-1}{\omega+1}\right)^2(1+i).
\]
Therefore
\[
       \sigma=\frac{2(1+\omega^2)}{(1+\omega)^2}(1+i),
\]
so equality holds in the upper endpoint of \eqref{eq:sharp-scalar-sigma}, and equality also holds in \eqref{eq:sharp-scalar-q}.  This proves sharpness.
\end{proof}

\section{Proof of Theorems~\ref{thm:main} and \ref{thm:sharp}}\label{sec:main-proof}
We now prove Theorem~\ref{thm:main}.  The upper estimate uses the scalar bound from Proposition~\ref{prop:scalar-higham}.  The lower estimate uses the Loewner domination in Lemma~\ref{lem:schur-loewner-lower}.

\begin{proof}[Proof of Theorem~\ref{thm:main}]
Let
\[
M_0:=\max_{i,j}|a_{ij}|.
\]
By Lemma~\ref{lem:diagonal-max}, we know
\begin{equation*}
	M_0=\max_i |a_{ii}|.
\end{equation*}
Let
\[
m_B=\lambda_{\min}(B),\quad M_B=\lambda_{\max}(B),
\qquad
m_C=\lambda_{\min}(C),\quad M_C=\lambda_{\max}(C).
\]
Then\footnote{Here the diagonal maximality
	is essential.  Since \(A=B+iC\in\HH_n\), Lemma~\ref{lem:diagonal-max} gives
	\(\max_{i,j}|a_{ij}|=\max_i|a_{ii}|\).  For each diagonal entry,
	\(a_{ii}=b_{ii}+ic_{ii}\), and hence
	\[
	|a_{ii}|=\sqrt{b_{ii}^2+c_{ii}^2}.
	\]
	Moreover,
	\(b_{ii}=e_i^TBe_i\le\lambda_{\max}(B)=M_B\) and
	\(c_{ii}=e_i^TCe_i\le\lambda_{\max}(C)=M_C\).  Therefore
	\[
	|a_{ii}|\le \sqrt{M_B^2+M_C^2}.
	\]
	The second inequality in \eqref{eq:M0-eigen-bound} follows from
	\(\kappa(B)=M_B/m_B\le\omega\) and \(\kappa(C)=M_C/m_C\le\omega\), which imply
	\(M_B\le\omega m_B\) and \(M_C\le\omega m_C\).}
\begin{equation}\label{eq:M0-eigen-bound}
	M_0\le \sqrt{M_B^2+M_C^2}
	\le \omega\sqrt{m_B^2+m_C^2}.
\end{equation}

For $1\le k\le n-1$, let $S_k=A^{(k)}$ be the active matrix after $k$ elimination steps.  Equivalently,
\[
       S_k=A/A[I_k,I_k],
       \qquad I_k=\{1,\ldots,k\}.
\]
The leading principal block $A[I_k,I_k]$ is Higham and hence nonsingular.  By Lemma~\ref{lem:inverse-schur}, $S_k$ is a Higham matrix; in particular Gaussian elimination without pivoting is well defined.  Write
\[
       S_k=P_k+i Q_k,
       \qquad P_k,Q_k>0.
\]
Lemma~\ref{lem:schur-loewner-lower}, applied to the Schur complement of the principal block indexed by $I_k$, gives
\[
       P_k\ge m_B I,
       \qquad
       Q_k\ge m_C I.
\]
Hence every diagonal entry of $S_k$ satisfies
\[
       |(S_k)_{rr}|
       =\sqrt{(P_k)_{rr}^2+(Q_k)_{rr}^2}
       \ge \sqrt{m_B^2+m_C^2}.
\]
Together with \eqref{eq:M0-eigen-bound}, this yields
\[
       \max_{r,s}|(S_k)_{rs}|
       \ge \max_r |(S_k)_{rr}|
       \ge \frac{M_0}{\omega}.
\]
Therefore $\rho_{n,k}(A)\ge1/\omega$ for every $1\le k\le n-1$.

It remains to prove the upper estimate.  By Lemma~\ref{lem:diagonal-max} applied to the active Higham matrix $S_k$,
\begin{equation}\label{eq:Skdiagmax}
       \max_{r,s}|(S_k)_{rs}|=\max_r |(S_k)_{rr}|.
\end{equation}
Fix $k$ and a remaining index $j>k$.  The $j$th diagonal entry of $S_k$, with indices written in the original numbering, is
\begin{equation*}
       (S_k)_{jj}=a_{jj}-a_{jI_k}A[I_k,I_k]^{-1}a_{I_kj}.
\end{equation*}
Consider the principal matrix
\[
       A[I_k\cup\{j\},I_k\cup\{j\}]
       =\begin{bmatrix}
          A[I_k,I_k]&a_{I_kj}\\
          a_{jI_k}&a_{jj}
        \end{bmatrix}.
\]
It is a Higham matrix.  Its real and imaginary parts are principal submatrices of $B$ and $C$, respectively, so their condition numbers are at most $\omega$ by Cauchy interlacing theorem.  Proposition~\ref{prop:scalar-higham} applied to this block gives
\[
       |(S_k)_{jj}|\le\frac{2(1+\omega^2)}{(1+\omega)^2} |a_{jj}|\le\frac{2(1+\omega^2)}{(1+\omega)^2} M_0.
\]
Using \eqref{eq:Skdiagmax}, we obtain
\[
       \max_{r,s}|(S_k)_{rs}|\le\frac{2(1+\omega^2)}{(1+\omega)^2} M_0
       \qquad (1\le k\le n-1).
\]
Thus $\rho_{n,k}(A)\le\frac{2(1+\omega^2)}{(1+\omega)^2}$ for every $k$, and taking the maximum over $1\le k\le n-1$ gives \eqref{eq:main-bound}.
\end{proof}

We prove Theorem~\ref{thm:sharp} by explicit examples.

\begin{proof}[Proof of Theorem~\ref{thm:sharp}]
For the lower constant, fix $n\ge2$ and set
\[
       D_\omega=\diag(\omega,1,\ldots,1),
       \qquad
       A_{\omega,n}^{\rm low}=(1+i)D_\omega .
\]
Then $B=C=D_\omega$ are real symmetric positive definite and
\[
       \kappa(B)=\kappa(C)=\omega.
\]
The matrix is diagonal, so Gaussian elimination produces no fill-in.  Its initial largest entry has modulus $\omega\sqrt2$, while after every step $1\le k\le n-1$ the active matrix has largest entry of modulus $\sqrt2$.  Therefore
\[
       \rho_{n,k}(A_{\omega,n}^{\rm low})=\rho_n(A_{\omega,n}^{\rm low})=\frac1\omega
       \qquad (1\le k\le n-1).
\]
This proves sharpness of the lower constant at every fixed stage.

For the upper constant, we set
\begin{equation*}
       A_\omega=
       \begin{bmatrix}
       1+i&\dfrac{\omega-1}{\omega+1}(1-i)\\[1mm]
       \dfrac{\omega-1}{\omega+1}(1-i)&1+i
       \end{bmatrix}.
\end{equation*}
Then
\[
       \Ree A_\omega=
       \begin{bmatrix}1&\dfrac{\omega-1}{\omega+1}\\[1mm]
       \dfrac{\omega-1}{\omega+1}&1\end{bmatrix},
       \qquad
       \Imm A_\omega=
       \begin{bmatrix}1&-\dfrac{\omega-1}{\omega+1}\\[1mm]
       -\dfrac{\omega-1}{\omega+1}&1\end{bmatrix}.
\]
Both are real symmetric positive definite and both have eigenvalues
\[
       1-\frac{\omega-1}{\omega+1}=\frac2{\omega+1},
       \qquad
       1+\frac{\omega-1}{\omega+1}=\frac{2\omega}{\omega+1}.
\]
Thus
\[
       \kappa(\Ree A_\omega)=\kappa(\Imm A_\omega)=\omega.
\]
The largest initial entry has modulus $|1+i|=\sqrt2$.  After the first elimination step, the trailing scalar is
\[
       (1+i)-
       \frac{\left(\frac{\omega-1}{\omega+1}\right)^2(1-i)^2}{1+i}
       =\frac{2(1+\omega^2)}{(1+\omega)^2}(1+i).
\]
Consequently
\[
       \rho_2(A_\omega)=\frac{2(1+\omega^2)}{(1+\omega)^2}.
\]
This proves sharpness of the upper constant for every $1\le\omega<\infty$.
\end{proof}

\begin{rem}
For each fixed Higham matrix the number $\omega=\max\{\kappa(B),\kappa(C)\}$ is finite, so Theorem~\ref{thm:main} gives
\[
       \rho_n(A)\le \frac{2(1+\omega^2)}{(1+\omega)^2}<2
\]
whenever $\omega>1$, and equality $\rho_n(A)=1$ when $\omega=1$.  The sharp examples above satisfy $\rho_2(A_\omega)=\frac{2(1+\omega^2)}{(1+\omega)^2}\to2$ as $\omega\to\infty$, which proves that the unconditioned constant $2$ is the sharp uniform upper constant over the whole Higham class $\HH_n$.
\end{rem}

\section{The case of accretive-dissipative matrices}\label{sec:ad}
The preceding scalar estimate has a Hermitian analogue.  This is useful for separating what is genuinely Higham-specific from what remains true for the larger accretive-dissipative class.
We first record the hereditary property needed for Gaussian elimination
in this class.

\begin{lem}\label{prop:ad-inverse-schur}
If $G=P+i Q\in\AD_n$, then 
\[
       G^{-1}=X-i Y,
       \qquad X,Y>0 .
\]
Consequently, every Schur complement of a principal block of a matrix in $\AD_n$ again belongs to the corresponding accretive-dissipative class.
\end{lem}

\begin{proof}
The inverse formula was established in \cite[Lemma~1]{Ikramov2004}
and is also reproduced in \cite[Lemma~2]{XueHu2015}.  The
Schur-complement closure was obtained in
\cite[Property~6]{GeorgeIkramov2005} and is also recorded in
\cite[Lemma~1]{XueHu2015}.
\end{proof}
Now, we present a sharp scalar estimate in the accretive-dissipative class $\AD_n$.
\begin{prop}\label{prop:scalar-ad}
Let
\[
M=\begin{bmatrix}G&b+i c\\ b^*+i c^*&\alpha\end{bmatrix}
=\begin{bmatrix}P&b\\ b^*&\beta\end{bmatrix}
+i\begin{bmatrix}Q&c\\ c^*&\gamma\end{bmatrix}
\in\AD_{m+1},
\]
where $P,Q$ are Hermitian positive definite and $\alpha=\beta+i\gamma$.  Suppose that the real and imaginary parts of $M$ have condition numbers at most $\omega$.  Then
\begin{equation}\label{eq:scalar-ad-q}
       \left|(b^*+i c^*)G^{-1}(b+i c)\right|
       \le \left(\frac{\omega-1}{\omega+1}\right)^2 |\alpha|.
\end{equation}
Consequently the scalar Schur complement
\[
       \sigma=\alpha-(b^*+i c^*)G^{-1}(b+i c)
\]
satisfies
\begin{equation}\label{eq:scalar-ad-sigma}
       \frac{4\omega}{(1+\omega)^2}|\alpha|
       \le |\sigma|
       \le \frac{2(1+\omega^2)}{(1+\omega)^2} |\alpha|.
\end{equation}
The constants are sharp already in the Higham sub-class.
\end{prop}

\begin{proof}
Choose a nonsingular complex matrix $R$ such that
\[
       R^*PR=I,
       \qquad
       R^*QR=D=\diag(d_1,\ldots,d_m),
       \qquad d_j>0.
\]
Write
\[
       R^*b=x,
       \qquad
       R^*c=y,
       \qquad x,y\in\mathbb C^m.
\]
Then
\begin{equation}\label{eq:qsum-ad}
       q:=(b^*+i c^*)G^{-1}(b+i c)
       =\sum_{j=1}^m
       \frac{(\overline{x_j}+i\overline{y_j})(x_j+i y_j)}{1+i d_j}.
\end{equation}
By Lemma~\ref{lem:kantorovich-block},
\[
       x^*x=b^*P^{-1}b\le \left(\frac{\omega-1}{\omega+1}\right)^2\beta,
       \qquad
       y^*D^{-1}y=c^*Q^{-1}c\le \left(\frac{\omega-1}{\omega+1}\right)^2\gamma .
\]
If $q=0$ there is nothing to prove.  If $q\ne0$, take $\phi=\arg q$ and apply Lemma~\ref{lem:twodim} term by term to \eqref{eq:qsum-ad}.  The same calculation as in the proof of Proposition~\ref{prop:scalar-higham} gives
\[
       |q|\le \left(\frac{\omega-1}{\omega+1}\right)^2
       \bigl(\beta\cos\vartheta(\phi)+\gamma\sin\vartheta(\phi)\bigr)
       \le \left(\frac{\omega-1}{\omega+1}\right)^2|\alpha|.
\]
Thus \eqref{eq:scalar-ad-q} holds, and \eqref{eq:scalar-ad-sigma} follows from the triangle and reverse triangle inequalities.  Sharpness follows from the Higham matrices $A_\omega^-$ and $A_\omega^+$ in \eqref{eq:scalar-extremal-families}.
\end{proof}
We obtain the following diagonal Schur-complement bounds for accretive-dissipative matrices.
\begin{cor}\label{cor:ad-diagonal}
Let $A=B+i C\in\AD_n$ and suppose
$
       \max\{\kappa(B),\kappa(C)\}\le\omega .
$
Then every diagonal entry produced by Gaussian elimination without pivoting satisfies
\begin{equation}\label{eq:ad-diagonal-growth}
       \frac{4\omega}{(1+\omega)^2}|a_{jj}|
       \le |a_{jj}^{(k)}|
       \le \frac{2(1+\omega^2)}{(1+\omega)^2} |a_{jj}|,
       \qquad j>k.
\end{equation}
Both constants in \eqref{eq:ad-diagonal-growth} are sharp for every $1\le\omega<\infty$.
\end{cor}

\begin{proof}
The proof is the same as the diagonal part of the proof of Theorem~\ref{thm:main}, using Proposition~\ref{prop:scalar-ad} on the principal block indexed by $I_k\cup\{j\}$.  Sharpness follows from the two Higham matrices in \eqref{eq:scalar-extremal-families}.
\end{proof}

For full entrywise growth in the generalized class one cannot use Lemma~\ref{lem:diagonal-max}; it is false outside the complex symmetric real-part/imaginary-part setting.  The following elementary substitute provides entrywise control in the
accretive-dissipative class.

\begin{lem}[{\cite[Lemma~2.2]{GeorgeIkramovKucherov2002}}]\label{lem:ad-entry}
If $A=(a_{ij})=B+i C\in\AD_n$, then
\begin{equation*}
       \max_{i,j}|a_{ij}|
       \le \sqrt2\,\max_i |a_{ii}|.
\end{equation*}
The factor $\sqrt2$ is sharp.
\end{lem}

Next, we prove Theorem~\ref{thm:generalized}.

\begin{proof}[Proof of Theorem~\ref{thm:generalized}]
The upper bound is the standard diagonal-to-entry reduction.  Accretive-dissipativity is hereditary under Schur complementation, so every active matrix belongs to some $\AD_m$.  By Corollary~\ref{cor:ad-diagonal}, every active diagonal entry is bounded above by
\[
       \frac{2(1+\omega^2)}{(1+\omega)^2}\max_i|a_{ii}|,
\]
which is at most
\[
       \frac{2(1+\omega^2)}{(1+\omega)^2}\max_{i,j}|a_{ij}|.
\]
Lemma~\ref{lem:ad-entry}, applied to each active matrix, gives the asserted upper bound.

It remains to prove the lower bound.  Let
\[
       M_0=\max_{i,j}|a_{ij}|,
       \qquad
       m_B=\lambda_{\min}(B),\quad m_C=\lambda_{\min}(C).
\]
For a Hermitian positive definite matrix $H$ with eigenvalues in $[m,M]$, every diagonal entry is at most $M$ and every off-diagonal entry has modulus at most $(M-m)/2$.  Hence, if $M_B=\lambda_{\max}(B)$ and $M_C=\lambda_{\max}(C)$, then every diagonal entry of $A$ has modulus at most
\[
       \sqrt{M_B^2+M_C^2}
       \le \omega\sqrt{m_B^2+m_C^2},
\]
whereas every off-diagonal entry satisfies
\[
       |a_{ij}|\le |b_{ij}|+|c_{ij}|
       \le \frac{M_B-m_B+M_C-m_C}{2}
       \le \frac{\omega-1}{2}(m_B+m_C)
       \le \omega\sqrt{m_B^2+m_C^2}.
\]
Thus
\begin{equation}\label{eq:ad-initial-max-lower-proof}
       M_0\le \omega\sqrt{m_B^2+m_C^2}.
\end{equation}
Fix $1\le k\le n-1$ and write the active matrix as $S_k=R_k+i T_k$.  Lemma~\ref{lem:schur-loewner-lower} gives
\[
       R_k\ge m_B I,
       \qquad
       T_k\ge m_C I.
\]
Consequently every diagonal entry of $S_k$ has modulus at least $\sqrt{m_B^2+m_C^2}$.  Combining this with \eqref{eq:ad-initial-max-lower-proof} gives
\[
       \max_{i,j}|a_{ij}^{(k)}|\ge \frac{M_0}{\omega}
       \qquad (1\le k\le n-1),
\]
and hence $\rho_n(A)\ge1/\omega$.

Sharpness of the lower bound follows from the diagonal example
\[
       A=(1+i)\diag(\omega,1,\ldots,1),
\]
for which $\kappa(B)=\kappa(C)=\omega$ and $\rho_n(A)=1/\omega$.
\end{proof}
Next, we illustrate what is sharp in the generalized setting $\AD_n$.
\begin{rem}\label{rem:ad-sharp}
The scalar and diagonal Schur-complement constants in Proposition~\ref{prop:scalar-ad} and Corollary~\ref{cor:ad-diagonal} are sharp for every $\omega$, and the full post-elimination lower bound in Theorem~\ref{thm:generalized} is also sharp.  The extra factor $\sqrt2$ in the full entrywise upper bound comes only from the loss of diagonal maximality in the accretive-dissipative class.  Lemma~\ref{lem:ad-entry} shows that this entry-versus-diagonal obstruction is real, but Theorem~\ref{thm:generalized} should not be read as claiming that the resulting full-entry upper constant is optimal.
\end{rem}

\section*{Acknowledgments} 
The author is grateful to his advisor, Professor Minghua Lin, for his encouragement and support.
 This work is supported by the China Scholarship Council, the Young Elite Scientists Sponsorship Program for PhD Students (China Association for Science and Technology), and the Fundamental Research Funds for the Central Universities at Xi'an Jiaotong University (Grant No.~xzy022024045).

\appendix

\section{Drury's determinant proof of Higham's conjecture}\label{app:drury}
This appendix records the determinant route due to Drury \cite{Drury2013}.  It is independent of the direct proof above and is included for comparison.

For a matrix \(A\in\mathbb C^{n\times n}\), its numerical range is defined by
\[
W(A)=\{x^*Ax:\ x\in\mathbb C^n,\ \|x\|_2=1\}.
\]
For \(0\le\alpha<\pi/2\), write
\[
S_\alpha
=
\{z\in\mathbb C:\ \Ree z>0,\ 
|\Imm z|\le (\tan\alpha)\Ree z\}.
\]
We say that \(A\) has numerical range contained in a sector of half-angle
\(\alpha\) if, after multiplication by a unimodular scalar, its numerical range
is contained in \(S_\alpha\); that is, if there exists \(\theta\in\mathbb R\)
such that
\[
W(e^{i\theta}A)\subset S_\alpha .
\]

We recall Drury's sectorial Fischer-type determinant inequality
\cite[Theorem~1]{Drury2013}.  We state it in a form adapted to the
notation used below.

\begin{thm}[{\cite[Theorem~1]{Drury2013}}]\label{thm:drury-fischer}
	Let
$
	A=
	\begin{bmatrix}
		A_{11}&A_{12}\\
		A_{21}&A_{22}
	\end{bmatrix}\in \mathbb C^{n\times n}
$
	have the numerical range is contained in a
	sector of half-angle \(\alpha\), where
	\(A_{11}\in\mathbb C^{p\times p}\) and
	\(A_{22}\in\mathbb C^{q\times q}\).  Put
$
	m=\min\{p,q\}.
$
	If
$
	0<m\alpha<\pi/2,
$
	then
	\[
	|\det A|
	\le
	\sec^2(m\alpha)\,
	|\det A_{11}|\,|\det A_{22}|.
	\]
\end{thm}

Drury~\cite[p.~3130, last paragraph]{Drury2013} remarked that every
accretive-dissipative matrix \(A=B+iC\), with \(B,C>0\), has numerical range
in the open first quadrant and hence, after multiplication by a unimodular
scalar, in a sector of half-angle \(\pi/4\). It follows from his sectorial determinant
inequalities that
\[
\rho_n(A)\le 2\sqrt2
\]
for accretive-dissipative matrices.  In the scalar-block case \(p=1\),
his Theorem~1.1 also gives
\[
|\det A|\le 2\,|a_{11}|\,|\det A_{22}|,
\]
which Drury identified as the determinantal form related to Higham's
conjecture.  We next describe this route in detail, and then explain how
the same idea yields the strict bound for each fixed Higham matrix.

\subsection{Drury's determinant route}

Let \(I=\{1,\ldots,k\}\), and fix a remaining index \(j>k\).  Consider the
principal submatrix indexed by \(\{j\}\cup I\), ordered with the scalar index
\(j\) first:
\[
A[\{j\}\cup I,\{j\}\cup I]
=
\begin{bmatrix}
	a_{jj} & a_{jI}\\
	a_{Ij} & A[I,I]
\end{bmatrix}.
\]
The scalar Schur complement with respect to the block \(A[I,I]\) is precisely
the active diagonal entry produced after the first \(k\) steps of Gaussian
elimination:
\[
a_{jj}^{(k)}
=
a_{jj}-a_{jI}A[I,I]^{-1}a_{Ij}.
\]
Equivalently,
\[
a_{jj}^{(k)}
=
\frac{\det A[\{j\}\cup I,\{j\}\cup I]}{\det A[I,I]}.
\]

Assume that the numerical range of \(A\) is contained, after a unimodular
rotation, in a sector of half-angle \(\alpha\).  Since principal submatrices
inherit the same numerical-range inclusion, Theorem~\ref{thm:drury-fischer} applied to the above principal submatrix, with the scalar block
as the first block, gives the case \(p=1\).  Hence
\begin{equation}\label{eq:drury-diag}
	|a_{jj}^{(k)}|
	\le
	\sec^2\alpha\, |a_{jj}|.
\end{equation}

If \(A=B+iC\in\HH_n\), then \(e^{-i\pi/4}A\) has numerical range contained
in the sector \(S_{\pi/4}\).  Taking \(\alpha=\pi/4\) in
\eqref{eq:drury-diag} gives
\begin{equation*}
	|a_{jj}^{(k)}|\le 2|a_{jj}|.
\end{equation*}
Moreover, Higham matrices are closed under Schur complementation, and each
active Higham matrix has its largest entry in modulus on the diagonal; see Lemmas~\ref{lem:inverse-schur} and \ref{lem:diagonal-max}.
Therefore
\[
\max_{r,s}|a_{rs}^{(k)}|
=
\max_r |a_{rr}^{(k)}|
\le
2\max_r |a_{rr}|
=
2\max_{r,s}|a_{rs}|,
\]
which gives
\[
\rho_n(A)\le 2
\]
for Higham matrices.

For general accretive-dissipative matrices, the diagonal maximality property
is no longer available.  Instead one uses the entry-versus-diagonal estimate
in Lemma~\ref{lem:ad-entry}.  Combining \eqref{eq:drury-diag}, with
\(\alpha=\pi/4\), and Lemma~\ref{lem:ad-entry} yields
\[
\rho_n(A)\le 2\sqrt2 .
\]

\subsection{A refined form of Drury's route}

The preceding argument uses only the uniform sector \(S_{\pi/4}\).  For a
fixed Higham matrix, one can sharpen this observation.  Let
\(A=B+iC\in\HH_n\), and set
\[
T=e^{-i\pi/4}A=P+iQ,
\qquad
P=\frac{B+C}{\sqrt2},
\qquad
Q=\frac{C-B}{\sqrt2}.
\]
Then \(P>0\), and
\[
P+Q=\sqrt2\,C>0,
\qquad
P-Q=\sqrt2\,B>0.
\]
Hence
\[
-P<Q<P
\]
in the Loewner order.  Therefore
\[
H=P^{-1/2}QP^{-1/2}
\]
is Hermitian and satisfies
\[
-I<H<I.
\]
Since the dimension is finite, all eigenvalues of \(H\) lie strictly inside
\((-1,1)\).  Thus
\[
\delta_A:=\|H\|_\infty<1.
\]
For every nonzero vector \(x\), we have
\[
\frac{|\Imm x^*Tx|}{\Ree x^*Tx}
=
\frac{|x^*Qx|}{x^*Px}
\le
\delta_A .
\]
Consequently
\[
W(T)\subset S_{\alpha_A},
\qquad
\alpha_A:=\arctan\delta_A<\frac{\pi}{4}.
\]

Applying \eqref{eq:drury-diag} with this sharper angle \(\alpha_A\) gives,
for every active diagonal entry,
\[
|a_{jj}^{(k)}|
\le
\sec^2(\alpha_A)|a_{jj}|
=
(1+\delta_A^2)|a_{jj}|.
\]
Since \(1+\delta_A^2<2\), and since Higham matrices and their Schur
complements have their largest entry in modulus on the diagonal, it follows
that
\[
\rho_n(A)
\le
1+\delta_A^2
<2.
\]
Thus Drury's sectorial determinant method can be refined to recover the
strict Higham bound for each fixed matrix.  The resulting constant, however,
depends on \(A\) through
\[
\delta_A
=
\left\|
(B+C)^{-1/2}(C-B)(B+C)^{-1/2}
\right\|_\infty.
\]
In contrast, the main result of the present paper gives instead the sharp
condition-number-dependent estimate
\[
\rho_n(A)
\le
\frac{2(1+\omega^2)}{(1+\omega)^2},
\qquad
\omega=\max\{\kappa(B),\kappa(C)\},
\]
together with matching extremal examples for every fixed \(\omega\).
\end{document}